\documentclass[12pt]{amsart}

\usepackage{amsmath,amsthm,amsfonts,amscd,amssymb,amsopn,enumerate}
\usepackage{eucal}
\usepackage{mathrsfs,color,hyperref,yhmath}
\usepackage{wasysym}
\usepackage[all]{xy}
\usepackage{fullpage}
\usepackage{tikz-cd}

\newcommand{\mbb}[1]{\mathbb #1}

\newcommand{\mc}[1]{\mathcal #1}

\newcommand{\wh}{\widehat}
\newcommand{\Spec}{\operatorname{Spec}}

\newcommand{\customdiagram}[4]{
\fbox{\xymatrix @C=.33cm @R=-0.3cm {
 & {#2} \ar@/^0.4pc/[dr] \\
{#1} \ar@/^0.4pc/[ur] \ar@/_0.4pc/[dr] & & {#4} \\
 & {#3} \ar@/_0.4pc/[ur]
}}
}

\newcommand{\vardiagram}[2]{
\fbox{\xymatrix @C=.33cm @R=-0.3cm {
 & {#2}_1 \ar@/^0.4pc/[dr] \\
{#2}_0 \ar@/^0.4pc/[ur] \ar@/_0.4pc/[dr] & & {#1} \\
 & {#2}_2 \ar@/_0.4pc/[ur]
}}
}

\newcommand{\mvardiagram}[2]{
\mbox{\xymatrix @C=.33cm @R=-0.3cm {
 & {#2}_1 \ar@/^0.4pc/[dr] \\
{#2}_0 \ar@/^0.4pc/[ur] \ar@/_0.4pc/[dr] & & {#1} \\
 & {#2}_2 \ar@/_0.4pc/[ur]
}}
}

\newcommand{\varPdiagram}[2]{
\fbox{\xymatrix @C=.33cm @R=-0.3cm {
 & {#2}_\UU \ar@/^0.4pc/[dr] \\
{#2}_\BB \ar@/^0.4pc/[ur] \ar@/_0.4pc/[dr] & & {#1} \\
 & {#2}_\PP \ar@/_0.4pc/[ur]
}}
}

\newcommand{\antidiagram}[1]{
\fbox{\xymatrix @C=.33cm @R=-0.3cm {
 & {#1}_1 \ar@/_0.4pc/[dl] \\
{#1}_0 & & {#1} \ar@/_0.4pc/[ul] \ar@/^0.4pc/[dl] \\
 & {#1}_2 \ar@/^0.4pc/[ul]
}}
}

\newcommand{\coc}{\oper{coc}}
\newcommand{\Inn}{\oper{Inn}}

\theoremstyle{plain}
\newtheorem{thm}{Theorem}[section]

\newtheorem{lemma}[thm]{Lemma}

\newtheorem{prop}[thm]{Proposition}

\newtheorem*{thm*}{Theorem}
\newtheorem*{rem*}{Remark}
\newtheorem*{lem*}{Lemma}

\newtheorem*{cor*}{Corollary}
\newtheorem*{prop*}{Proposition}

\theoremstyle{definition}

\theoremstyle{remark}
\newtheorem{rem}[thm]{Remark}

\newcommand{\sheaf}[1]{\mathscr{#1}}

\newcommand{\PP}{\sheaf{P}}
\newcommand{\UU}{\sheaf{U}}
\newcommand{\BB}{\sheaf{B}}

\newcommand{\oper}[1]{\operatorname{#1}}

\newcommand{\cha}{\oper{char}}

\usepackage[OT2,T1]{fontenc}

\def\<{\left<}
\def\>{\right>}

\DeclareSymbolFont{cyrletters}{OT2}{wncyr}{m}{n}
\DeclareMathSymbol{\Sha}{\mathalpha}{cyrletters}{"58}
\title{Local-global principles for the existence of Levi factors}
\author{David Harbater, Julia Hartmann, and George McNinch}
\date{\today}
\begin{document}
\maketitle

\begin{abstract}
We discuss local-global principles for the existence of Levi factors (i.e., complements to the unipotent radical) for linear algebraic groups over one-variable function fields. We give examples of disconnected groups that fail the local-global principle, and prove a strong local-global principle in the presence of Levi descent.
\end{abstract}

\section{Introduction}
If $F$ is a field of characteristic zero, the canonical homomorphism from a linear algebraic group $G$ over $F$ to the quotient $G/R_u(G)$ by its unipotent radical admits a splitting (e.g., see \cite[Section~3.1]{McNinch10}). Every such splitting gives rise to a {\em Levi factor} of $G$, i.e., an $F$-subgroup that maps isomorphically to $G/R_u(G)$. If $F$ has positive characteristic, the situation is more complicated. First: The unipotent radical (over the algebraic closure) might not be defined over $F$. Second: If it is defined over $F$, the group $G$ still might not admit a Levi factor. In fact, the latter can happen even when $F$ is algebraically closed (see, e.g., \cite[Proposition~A.6.4]{CGP}; another source of examples comes from reductive groups $M$ with $H^2(M,V)$ nonzero for some linear representation~$V$). 

The existence (or nonexistence) and conjugacy of Levi factors was studied for example in \cite{Humphreys}, \cite{McNinch10}, \cite{McNinch13}, and \cite{McNinch} (see also the references there).

In this note we study local-global principles for the existence of Levi factors.  Generally speaking, a local-global principle asserts that a given algebraic property holds over a field $F$ provided that it holds over a set of overfields, such as the completions of $F$ at all discrete valuations.  Classical examples of local-global principles over global fields (in particular, number fields) include the Hasse-Minkowski theorem concerning isotropy of quadratic forms, and the theorem of Albert-Brauer-Hasse-Noether concerning the splitting of central simple algebras.  Here, we investigate the validity of local-global principles for the existence of Levi factors for linear algebraic groups over one-variable function fields, assuming that the unipotent radical is defined over the function field.  That is, we ask the following question: If a linear algebraic group~$G$ over such a function field $F$ (with $R_u(G)$ defined over $F$) admits a Levi factor over each completion $F_v$ at a discrete valuation $v$ on $F$, does it admit a Levi factor over~$F$? We show that the answer in general is negative (Proposition~\ref{example}, Proposition~\ref{elliptic}), but that a strong local-global principle does hold in the presence of Levi descent (Theorem~\ref{LGP}).

\section{Preliminaries}

By a {\em linear algebraic group} we mean a smooth affine group scheme of finite type over a field~$F$; any such group  embeds as a closed subgroup of $\operatorname{GL}_{n}$ for a suitable $n$ (\cite{Springer}, Theorem~2.3.7). An element $g$ is called {\em unipotent} if $g-1$ is nilpotent (viewed as an endomorphism of $F^n$; this is independent of the embedding). A linear algebraic group is {\em unipotent} if every element of~$G$ defined over a (chosen) algebraic
  closure of~$F$ is unipotent. Equivalently, $G$ is $F$-isomorphic to a closed subgroup of the upper triangular unipotent $n\times n$ matrices for some~$n$. Given a linear algebraic group $G$ over an algebraically closed field~$F$, its {\em unipotent radical} $R_u(G)$ is the largest connected normal unipotent subgroup. A linear algebraic group $G$ over an algebraically closed field~$F$ is called {\em reductive} if $R_u(G)$ is trivial. The subgroup $R_u(G)$ can also be characterized as the smallest connected normal subgroup so that $G/R_u(G)$ is reductive. If $G$ is defined over a field $F$ with algebraic closure $\bar F$, we say the unipotent radical of $G$ is {\em defined over $F$} if the largest connected normal unipotent $F$-subgroup $U$ of $G$ induces $R_u(G_{\bar F})$; i.e., $U_{\bar F}=R_u(G_{\bar F})$, where $U_{\bar F}$ and $G_{\bar F}$ denote the base change of $U$ and $G$ to $\bar F$, respectively.  The unipotent radical is defined over $F$ exactly when  $G/U$ is reductive; and in this case we will write $R_u(G)$ for $U$.  By \cite[12.1.7(d)]{Springer}, the unipotent radical is always defined over $F$ when $F$ is perfect. 

Now let $G$ be a linear algebraic group over a field~$F$ such that its unipotent radical is defined over~$F$. We then have a short exact sequence
$$1\rightarrow R_u(G)\rightarrow G\overset{\pi}{\rightarrow} G/R_u(G)\rightarrow 1, $$
and it is natural to ask whether this sequence admits a splitting. 

A splitting to this sequence furnishes a reductive subgroup $M$ of $G$ which maps isomorphically to $G/R_u(G)$ under $\pi$. Any such subgroup is called a {\em Levi subgroup} or a {\em Levi factor}. In characteristic zero, Mostow (\cite{Mostow}, Theorem~7.1) proved that Levi subgroups exist and that any two of them are conjugate by an element of the unipotent radical, using Lie algebra methods. In positive characteristic, assuming the unipotent radical of $G$ is defined over $F$, the group $G$ may still not have a Levi factor. Moreover, two Levi factors of $G$ might not be (geometrically) conjugate. See \cite[Section~3]{McNinch10} for references and a discussion of known results. 

A different approach to these kinds of questions is taken in \cite{CGP}, where the authors define the unipotent radical of an $F$-group $G$ to be the largest connected normal unipotent $F$-subgroup (whether or not it induces the unipotent radical over $\bar F$). The quotient group may then not be reductive, and it is called {\em pseudo-reductive}. Both notions agree over an algebraically closed field. In particular, the example given in \cite[Proposition~A.6.4]{CGP} of a linear algebraic group over an algebraically closed field that does not admit a Levi factor provides an example in our context as well.

\section{A counterexample to the local-global principle}
The aim of this section is to give examples of a one-variable function field $F$ over a field $K$, and a linear algebraic group $G$ over $F$ that has no Levi factor over $F$, such that the base change $G_{F_v}$ has a Levi factor over the $v$-adic completion $F_v$ for every discrete valuation $v$ on $F$. Hence $G$ violates the local-global principle for the existence of Levi factors. If some such~$G$ exists, we say that the local-global principle for Levi factors over $F$ fails. As noted earlier, the field $F$ (or equivalently, $K$) is necessarily of positive characteristic.

The first example we provide here relies on a construction in \cite[Section~8]{McNinch}, over a field $F$ with $\cha(F)>2$.  There, for every element $f \in F$, a linear algebraic $F$-group $E_f$ is defined, given by $E_f = H \times_{\mbb G_a} \mu_t$, where $H$ is a certain nonabelian central extension of $V = \mbb G_a^2$ by $\mbb G_a$, and $\mu_t$ is a certain extension of $(\mbb Z/p\mbb Z)^2$ by $\mbb G_a$ that depends on the choice of $f$.  
(See \cite[p.~395]{McNinch} for details).  We write $G_f = E_{-f}$.  The following proposition is then proved there, at \cite[Proposition~8.1]{McNinch}:

\begin{prop}\label{failure-criterion}
With the above notation, let $P(Z)=Z^p-Z-f\in F[Z]$, and let $L/F$ be a finite extension of fields. Then the base change $(G_f)_L$ of $G_f$ to $L$ has a Levi factor if and only if $P$ has a root in $L$.
\end{prop}

Since the polynomial $P$ in the above proposition is an Artin-Schreier polynomial, if it has a root in some extension $E$ of $F$, then $P$ is already split in $E$ (the other roots are ${\mathbb F}_p$-translates). 
Hence in order to use the above proposition to find $G$ such that the local-global principle for the existence of Levi factors fails, we need to find a function field $F$ and an element $f\in F$ such that the field extension $L/F$ defined by the Artin-Schreier polynomial $P(Z)=Z^p-Z-f$ is not trivial, but $P$ defines the trivial extension over $F_v$ for every discrete valuation $v$ of $F$. The following lemma will be useful in identifying trivial (resp.\ nontrivial) Artin-Schreier extensions.

\begin{lemma}
Let $R$ be a complete local domain of characteristic $p>0$ with fraction field~$L$, and let $f\in L$. 
\begin{enumerate}
\item[(a)]\label{trivial} If $f$ is in the maximal ideal of $R$, the field extension of $L$ given by $Y^p-Y-f=0$ is trivial. 
\item[(b)]\label{nontrivial} If $f$ has valuation equal to ~$-1$, the field extension of $L$ given by $Y^p-Y-f=0$ is nontrivial. 
\end{enumerate}
\end{lemma}
\begin{proof}
If $f$ is in the maximal ideal of $R$, then $u:=-\sum\limits_{i=0}^\infty f^{p^i}$ is a well defined element of the complete ring $R$. This element satisfies $u^p-u=f$, and so the extension is trivial, proving the first assertion. 
For the second part, let $u$ be a nonzero element of $L$. If $u$ has nonnegative valuation, then so does $u^p-u$. On the other hand, if $u$ has valuation $m<0$, then $u^p-u$ has valuation $mp$. In summary, no element of the form $u^p-u$ has valuation $-1$, and so $Y^p-Y-f=0$ has no solution in $L$.
\end{proof}

\begin{prop}\label{example}
Let $k$ be a field of characteristic $p>0$, let $K=k((t))$, and let $F/K(x)$ be the degree $p$ Galois extension defined by the Artin-Schreier equation $Y^p-Y=\frac{t}{x}+\frac{t}{x-1}$. Let $P:=Z^p-Z-\frac{t}{x}\in F[Z]$. Then $P$ does not have a root in $F$ but does have a root in $F_v$ for every discrete valuation $v$ of $F$. Hence the local-global principle for Levi factors over $F$ fails.
\end{prop}

\begin{proof}
Lemma~\ref{nontrivial}(b) applied to $R=K[[x]]$ shows that $F/K(x)$ is in fact a degree $p$ extension. Similarly, $P$ defines a degree $p$ extension $E$ of $K(x)$. For the first assertion of the proposition, suppose for contradiction that $P$ has a root in $F$.  Then $E \subseteq F$, and this containment is an equality since both extensions $E$, $F$ have degree $p$ over $K(x)$.   
Lemma~\ref{trivial}(a) with $R=k[[t,x-1]]$ and $f=\frac{t}{x}$ implies that 
$E$ is contained in the fraction field $k((t,x-1))$ of $k[[t,x-1]]$, and hence it is contained in $k((t))((x-1)) = K((x-1))$.  
On the other hand, Lemma~\ref{nontrivial}(b), with $R=K[[x-1]]$ and $f=\frac{t}{x}+\frac{t}{x-1}$, shows that the compositum $FK((x-1))$ (in any common overfield) is a degree $p$ extension of $K((x-1))$.   Thus $F$ is not contained in $K((x-1))$.  Since $E$ is contained in $K((x-1))$ and $E=F$, this is a contradiction; hence $P$ has no root in $F$.

For the second assertion, let $v$ be a discrete valuation on $F$. Then $v_0:=v|_{K(x)}$ defines a discrete valuation on $K(x)$. By \cite{HHK:H1}, Proposition~7.4, there is a point $Q\in {\mathbb P}^1_k$ for which the fraction field $K(x)_Q$ of the complete local ring $\wh R_Q$ of ${\mathbb P}^1_{k[[t]]}$ at $Q$ is contained in $K(x)_{v_0}$ and hence in $F_v$. If $Q$ is not the point $x=0$ on ${\mathbb P}^1_k$, then $P$ defines the trivial extension of $K(x)_Q$ by Lemma~\ref{trivial}(a), and hence defines the trivial extension of $F_v$. If $Q$ is the point $x=0$, then $\wh R_Q=k[[t,x]]$. In its fraction field $K(x)_Q$, there is a root $w$ of $W^p-W=\frac{t}{x-1}$, again by Lemma~\ref{trivial}(a). Therefore, there is also a root $w$ in $F_v$. Meanwhile, $F_v$ also contains a root $y$ of $Y^p-Y=\frac{t}{x}+\frac{t}{x-1}$ (by definition of $F$). Let $z:=y-w$. Then $z^p-z=\frac{t}{x}$. Hence the extension is trivial over $F_v$.  

Consequently, $G_{\frac{t}{x}}$ violates the local-global principle for the existence of Levi factors, by 
Proposition~\ref{failure-criterion}.  Thus the assertion follows.
\end{proof}

\begin{rem}
\begin{enumerate}
\item Every complete discretely valued field $K$ of characteristic $p>0$ is of the form $k((t))$ as above, e.g. by the Cohen Structure Theorem (see \cite[Theorem~15]{Cohen}). 
\item By using \cite[Proposition~1.4]{comparison} (again applied to ${\mathbb P}^1_{k[[t]]}$) instead of \cite[Proposition~7.4]{HHK:H1} in the above proof, one finds that the conclusion of the above proposition remains true even if we consider the larger set $\Omega_F$ of all valuations on $F$ whose valuation ring contains $k[[t]]$. (Note that the set $\Omega_F$ contains all the discrete valuations on $F$ by \cite[Corollary~7.2]{HHK:H1}.)
\end{enumerate}
\end{rem}

Proposition~\ref{failure-criterion} above can also be used to produce other counterexamples. 

\begin{prop}\label{elliptic}
Let $K$ be an algebraically closed field of characteristic $p>0$.  Then for all but finitely many elliptic curves $C$ over $K$ (up to isomorphism), there is a nontrivial Artin-Schreier  extension of the function field $F$ of $C$ that splits over each completion of $F$.  As a consequence, the local-global principle for Levi factors over $F$ fails.
\end{prop}
\begin{proof}
First recall that to each elliptic curve~$C$ one may associate a Hasse invariant, which is zero for only finitely many isomorphism classes of elliptic curves (\cite[Chapter IV, Corollary 4.23]{Hartshorne}). For all other elliptic curves, $\mbb Z/p\mbb Z$ is a quotient of the algebraic fundamental group by \cite[Chapter IV, Exercise~4.8]{Hartshorne}. 
That is, there is a connected Galois finite \'etale cover $C' \to C$ whose Galois group over $C$ is $\mbb Z/p\mbb Z$.  This cover defines a Galois extension $E$ of the function field $F$ of $C$, which is Artin-Schreier since $F$ has characteristic $p$.  This is given by a polynomial of the form $Y^p-Y=f$ that has no root in $F$, for some $f \in F$.  

Now consider any discrete valuation $v$ on $F$ over $K$.  Since $K$ is algebraically closed, $K^\times$ and hence $v(K^\times) \subseteq \mbb Z$ is divisible. Since $v$ is discrete, it must be trivial on~$K$. Thus $v$ corresponds to a closed point $Q$ on $C$ which is necessarily a $K$-point (see \cite[Chapter I, Corollaries~6.6, 6.10, 6.12]{Hartshorne}). Since the cover is finite \'etale, the fiber over $Q$ consists of exactly $p$ distinct $K$-points.  For each such point $Q' \in C'$ over $P \in C$, Hensel's Lemma then implies that the inclusion of complete local rings $\wh{\mc O}_{C,Q} \subseteq \wh{\mc O}_{C',Q'}$ is trivial.  Thus the polynomial $Y^p-Y=f$ has a root in the fraction field $F_v$ of $\wh{\mc O}_{C,Q}$; i.e., the extension splits over $F_v$.

Thus $G_f$ violates the local-global principle for the existence of Levi factors by Proposition~\ref{failure-criterion}, and the result follows.
\end{proof}

\begin{rem}
The above proposition can be generalized to curves of higher genus, since for every $g \ge 1$, there is a dense open subset of the moduli space of genus $g$ curves over $K$ such that the corresponding curve has a Galois finite \'etale cover of degree $p$, corresponding to an unramified Artin-Schreier extension of the function field of the curve.
\end{rem}

\section{A strong local-global principle via Levi descent}

In contrast to the counterexamples given in the previous section, below we provide a condition that implies not only that a local-global principle for the existence of Levi factors holds, but also that it holds in a very strong form.  This condition relates to the behavior of Levi factors under base change.

If $E/F$ is a field extension, and if $G$ is a linear algebraic group over $F$ that has a Levi factor~$M$, then the base change $G_E$ also has a Levi factor; viz., $M_E$.  The converse does not hold in general, but it is known to hold in certain situations.  

We say that a linear algebraic group $G$ over a field $F$ {\it satisfies Levi descent} (resp., {\it away from $\cha(F)$}) if it has the following property: For every field $E$ containing $F$ and every finite Galois extension $L/E$ (resp., of degree not divisible by $\cha(F)$), if $G_L$ has a Levi factor then so does $G_E$.  For connected $F$-groups $G$ for which $R_u(G)$ is defined over $F$, there are {\em no} known examples where Levi descent fails. Remark~\ref{descent_remark} below discusses some cases when Levi descent is known to hold.

Now let $K$ be a field, and let $F$ be a one-variable function field over~$K$. In the case of $F$-groups $G$ that satisfy Levi descent, we then have a strong local-global principle for the existence of a Levi factor: The group $G$ has a Levi factor if and only if $G_{F_v}$ does for {\it some} completion $F_v$ with respect to a divisorial discrete valuation~$v$ (i.e., a valuation whose valuation ring is the local ring of a point on some regular projective connected $K$-curve with function field~$F$). Namely, we have the following theorem.

\begin{thm}\label{LGP}
Let $K$ be a field, and let $F$ be 
a one-variable function field over~$K$.
Let $G$ be a linear algebraic group over $F$ whose unipotent radical is defined over~$F$, and assume that $G$ satisfies Levi descent.  

Then the following are equivalent:
\begin{enumerate} [(a)]
\item $G$ has a Levi factor.
\item $G_{F_v}$ has a Levi factor for every completion $F_v$ at a discrete valuation~$v$ of~$F$.
\item $G_{F_v}$ has a Levi factor for some completion $F_v$ at a divisorial discrete valuation~$v$ of~$F$. 
\end{enumerate}
\end{thm}
\begin{proof}
Since the existence of a Levi factor is preserved under base change, and since (b) clearly implies~(c), we only need to show that (c) implies~(a). 
Let $U$ be the unipotent radical of $G$.
Let $v$ be the discrete valuation for which $G_{F_v}$ has a Levi factor. Since $v$ is divisorial, there exists a regular projective connected  $K$-curve~$C$ with function field $F$ and a point $Q$ of $C$ such that the completion $F_v$ is the fraction field of the complete local ring $\wh R_v:=\wh {\mc O}_{C,Q}$.  Consider the henselization $R_v^h := {\mc O}^h_{C,Q}$ of ${\mc O}_{C,Q}$, with $F_v^h$ its fraction field.  Thus $F_v$ is the completion of $F_v^h$.  
By \cite[Proposition~7.8.6]{EGA4}, $C$ is excellent, and so ${\mc O}_{C,Q}$ is an excellent discrete valuation ring. 

Let $\pi: G\rightarrow G/U$ be the canonical epimorphism to the quotient, with base change $\pi_v: G_{F_v}\rightarrow (G/U)_{F_v}\simeq G_{F_v}/U_{F_v}$. The existence of a Levi factor of $G_v$ is equivalent to the existence of a (homomorphic) section $s_v$ of $\pi_v$. The comorphism $s_v^*: F_v[G]\rightarrow F_v[G/U]=F_v^h[G/U]\otimes_{F_v^h} F_v$ is given by a finite amount of data, and so there is a finitely generated $F_v^h$-subalgebra $A \subset F_v$ 
to which $s_v^*$ descends.  The descended comorphism corresponds to a section $(G/U)_A \to G_A$
to the base change $\pi_A$ of $\pi$ to $A$ so that $s_A$ induces $s_v$.

Since the ring $A$ is a finitely generated algebra over the field $F_v^h$,
there is an isomorphism 
$\rho:F_v^h[X_1,\dots,X_r]/(f_1,\dots,f_n) \to A$ for some polynomials
$f_i \in F_v^h[X_1,\dots,X_r]$.  After multiplying each $X_i$ by a suitable power of the uniformizer of $\mc O_{C,Q}$, we may assume that each $f_j$ lies in $R_v^h[X_1,\dots,X_r]$ and that each $\bar x_i := \rho(X_i) \in A \subseteq F_v$ lies in $\wh R_v$.   
Let $\rho_0$ be the composition of 
$R_v^h[X_1,\dots,X_r]/(f_1,\dots,f_n)\to F_v^h[X_1,\dots,X_r]/(f_1,\dots,f_n)$ with the isomorphism $\rho$.
The kernel of $\rho_0$ is of the form $(f_1,\dots,f_n,f_{n+1},\dots,f_m)$ for some $m \ge n$ and some $f_{n+1},\dots,f_m \in R_v^h[X_1,\dots,X_r]$.  
After replacing $n$ by $m$, we may assume that $\rho_0$ is injective, 
defining an isomorphism $R_v^h[X_1,\dots,X_r]/(f_1,\dots,f_n) \to A_0$,
where $A_0$ is the image of $\rho_0$.
The elements $\bar x_1,\dots,\bar x_r \in A_0\subseteq \wh R_v$ 
satisfy the polynomials $f_j \in R_v^h[X_1,\dots,X_r]$, for $j=1,\dots,n$.  
Since $R_v^h$ is the henselization of the excellent discrete valuation ring ${\mc O}_{C,Q}$ and $\wh R_v$ is its completion, the Artin Approximation Theorem (see \cite[Theorem~1.10]{Art}) yields elements $x_1,\dots,x_r \in R_v^h$ that also satisfy those polynomials.  
By sending $X_i$ to $x_i$, we obtain an $R_v^h$-algebra homomorphism $R_v^h[X_1,\dots,X_r]/(f_1,\dots,f_n) \to R_v^h$.  Tensoring with $F_v^h$ yields an $F_v^h$-algebra homomorphism $F_v^h[X_1,\dots,X_r]/(f_1,\dots,f_n) \to F_v^h$, corresponding to an $F_v^h$-point on $\Spec(A)$.  By pushing forward with respect to this homomorphism (or equivalently, pulling back with respect to the scheme morphism $\Spec(F_v^h) \to \Spec(A)$), we obtain a morphism $s_v^h:(G/U)_{F_v^h}\rightarrow G_{F_v^h}$ which is a section of $\pi_{F_v^h}$ by functoriality
of pullback and since $s_A$ is a section of $\pi_A$.

The extension $F_v^h/F$ is separable algebraic. By an argument similar to the one above, the section $s_v^h$ is defined over a finitely generated $F$-subalgebra of $F_v^h$, and thus over a finite separable field extension $L$ of $F$.  That is, $G_L$ has a Levi factor. Replacing~$L$ by its Galois closure $E$ over $F$ and using that $G$ satisfies Levi descent then gives the desired result.
\end{proof}

We conclude this section with a summary of what is known about the condition of Levi descent.
\begin{rem}\label{descent_remark}
Suppose that $G$ is a linear algebraic group over a field $F$ such that the unipotent radical $U$ of $G$ is defined and split over $F$ (i.e., there is a sequence $1 = U_0 \subset U_1 \subset \cdots \subset U_m = U$ of closed connected normal $F$-subgroups of $U$ such that the successive quotients are each isomorphic to the additive group over $F$).  

Then $G$ satisfies Levi descent away from $\cha(F)$, by \cite[Theorem~3.4.3]{McNinch13}. Moreoever, $G$ satisfies Levi descent under each of the following hypotheses:

\begin{enumerate} [(i)]
\item (\cite[Theorem~1.6]{McNinch}) The invariant group scheme $U_\ell^{M_\ell}$ is trivial, and $H^1_{\coc}(M_\ell,U_\ell)=1$.

\item (\cite[Theorem~1.7]{McNinch}) $\Inn(U_\ell)^{M_\ell}$ is trivial, $H^1_{\coc}(M_\ell,\Inn(U_\ell))=1$, and $Z(U)$ is a vector group on which $G$ acts linearly.
\end{enumerate}

\noindent Here $Z(U)$ is the center of the unipotent radical $U \subseteq G$, and $H^1_{\coc}$ is the first cohomology set defined in \cite[Section~3]{McNinch}, where the allowable cochains are those that are morphisms of varieties.

\end{rem}

{\bf Acknowledgements.} The authors thank Florian Pop for helpful discussions.

\bigskip

\noindent {\bf Author information:}

\medskip

\noindent David Harbater, 
Department of Mathematics, University of Pennsylvania, Philadelphia, PA 19104-6395, USA\\
email: harbater@math.upenn.edu

\medskip

\noindent Julia Hartmann, 
Department of Mathematics, University of Pennsylvania, Philadelphia, PA 19104-6395, USA\\
email: hartmann@math.upenn.edu

\medskip

\noindent George McNinch,
Department of Mathematics, Tufts University, Medford, MA 02155, USA\\
email: george.mcninch@tufts.edu

\medskip

\noindent The authors were supported on DMS-2102987 and DMS-2402367 (DH and JH).

\end{document}